\documentclass[12pt]{amsart}
\pdfoutput=1
\usepackage{graphicx}
\usepackage{amssymb}
\usepackage{amsmath}
\usepackage{amsthm}
\usepackage{amscd}
\usepackage{epsfig}

\newtheorem{theorem}{Theorem}[section]

\newtheorem*{theorem*}{Theorem}

\newtheorem{proposition}[theorem]{Proposition}

\theoremstyle{definition}

\numberwithin{equation}{section}

\newcommand{\R}{\mathbb R}
\newcommand{\Z}{\mathbb Z}
\newcommand{\N}{\mathbb N}

\begin{document}

\title{On large deviations for amenable group actions}

\author [Dongmei Zheng, Ercai Chen and Jiahong Yang]{Dongmei Zheng, Ercai Chen and Jiahong Yang}

\address{School of Mathematical Sciences and Institute of Mathematics, Nanjing Normal University, Nanjing 210023, Jiangsu, P.R.China
 \& Department of Applied Mathematics, College of Science, Nanjing Tech University, Nanjing 211816, Jiangsu, P.R. China} \email{dongmzheng@163.com}

\address{School of Mathematical Sciences and Institute of Mathematics, Nanjing Normal University, Nanjing 210023, Jiangsu, P.R.China} \email{ecchen@njnu.edu.cn}

\address{School of Mathematical Sciences and Institute of Mathematics, Nanjing Normal University, Nanjing 210023, Jiangsu, P.R.China} \email{848115840@qq.com}

\thanks{}


\begin{abstract}
By proving an amenable version of Katok's entropy formula and handling the quasi tiling techniques, we establish large deviations bounds for countable discrete amenable group actions. This generalizes the classical results of Lai-Sang Young \cite{Y}.
\end{abstract}

\maketitle

\section{Introduction}
The large deviation theory focuses on convergence properties of stochastic systems. The stochastic systems usually can occur in
various stochastic processes, differential equations or dynamical systems. In the case of dynamical systems, the usual setting is a continuous transformation
on a compact metric space with a reference measure. Then one of interesting questions is to obtain the exponential convergence rate for the reference measure of some subset related to
Birkhoff average.

Many results on the large deviation theory for dynamical systems are related to entropy or pressure. Usually the desired exponential convergence rates are functions in terms of
entropy or pressure. And these results can have applications to many aspects such as statistical mechanics, non-uniformly hyperbolic systems, shift spaces, etc.(\cite{E,K,RY,Y}).

With the development of the theory of dynamical systems, $\Z^d-$ actions are considered as the generalization for
$\Z-$actions. Some large deviation results are proved for these situations \cite{EKW,Ki}. These results rely on the entropy and pressure theory for $\Z^d-$actions.
But now the theory of entropy and pressure goes much further. It is established for actions of groups such as amenable groups and sofic groups beyond $\Z^d$\cite{C,L,LY,O,OP,OW,ST,W}.
In this paper, we will consider the large deviation results of Lai-Sang Young \cite{Y} for actions of countable discrete amenable groups.

Throughout this paper, we let $(X,G)$  be a $G-$action topological dynamical system, where $X$ is a compact metric space with metric $d$ and $G$ a countable discrete amenable group.
A group $G$ is said to be {\it amenable} if there exists a sequence of finite subsets $\{F_n\}$ of $G$ which are asymptotically invariant, i.e.,
$$\lim_{n\rightarrow+\infty}\frac{|F_n\vartriangle gF_n|}{|F_n|}=0, \text{ for all } g\in G.$$
Such sequences are called {\it F{\o}lner sequences}.

For a finite subset $F$ in $G$, $\epsilon>0$ and a function $f$ of $X$, we denote the {\it Bowen ball} associated to $F$ with radius $\epsilon$ by
\begin{align*}
B_{F}(x,\epsilon)&=\{y\in X: d_{F}(x,y)<\epsilon\}\\
&=\{y\in X: d(gx,gy)<\epsilon, \text{ for any }g\in F\}
\end{align*}
and denote the summation and average of $f$ along $F$ by
\begin{align*}
  S_{F}f(x)=\sum_{g\in F}f(gx)
\end{align*}
and
\begin{align*}
  A_{F}f(x)=\frac{1}{|F|}\sum_{g\in F}f(gx)
\end{align*}
respectively.

Let $M(X)$, $M(X,G)$ and $E(X,G)$ denote the collection of Borel probability measures, the $G-$invariant and $G-$ergodic Borel probability measures of $X$ respectively.

A F{\o}lner sequence $\{F_n\}$ in $G$ is said to be {\it tempered} (see Shulman \cite{S}) if there exists a constant $C$ which is independent of $n$ such that
\begin{align*}
|\bigcup_{k<n}F_k^{-1}F_n|\le C|F_n|, \text{ for any }n.
\end{align*}

Let $(X,G,\mu)$ be a measure-theoretic $G-$action dynamical system where $G$ is a countable discrete amenable group and $\mu$ is a $G-$ergodic Borel probability measure on $X$.
The ergodic theorem(\cite{L,W}) states that, if $\{F_n\}$ is a tempered F{\o}lner sequence in $G$ and $f\in L^1(X,\mathcal{B},\mu)$, then
$$\lim_{n\rightarrow+\infty}A_{F_n}f(x)=\int_{X}f(x)\, d\mu,$$
almost everywhere and in $L^1$.

Let $\{F_n\}$ be a F{\o}lner sequence in $G$. For $\mu\in M(X)$, $\varphi\in C(X,\R)$ and $E\subset \R$, we will consider the exponential growth rate of
$\mu(\{x\in X: A_{F_n}\varphi(x)\in E\})$. Then we have the following large deviations results.

\begin{theorem}\label{th-1}
  Let $(X,G)$ be a compact metric $G-$action topological dynamical system and $G$ a discrete countable amenable group. Let $\{F_n\}$ be a
  tempered F{\o}lner sequence in $G$ with $\lim\limits_{n\rightarrow+\infty}\frac{|F_n|}{\log n}=\infty$. Assume $h_{top}(X,G)<\infty$.
  Then for any $\mu\in M(X)$, $\varphi\in C(X)$ and $c\in \R$,
  \begin{align*}
    \liminf_{n\rightarrow\infty}\frac{1}{|F_n|}\log \mu(\{x\in X: A_{F_n}\varphi(x)>c\})\ge \sup\{h_{\nu}(X,G)-h_{\mu}(\{F_n\};\nu)\},
  \end{align*}
  where the supremum is taken over all $\nu\in E(X,G)$ with $\int\varphi \, d\nu>c$.
\end{theorem}
We note here that $h_{top}(X,G)$ and $h_{\nu}(X,G)$are the topological and metric entropy of $(X,G)$ respectively and they do not depend on the choice of
the F{\o}lner sequences. $h_{\mu}(\{F_n\};\nu)$ is the $\nu-$relative entropy with respect to $\mu$ for F{\o}lner sequence $\{F_n\}$.
We will give the detailed definitions of these entropies in Section 2.

Now we denote by
\begin{align*}
  \mathcal{V}^+=\{\psi\in C(X,\R): & \exists C,\epsilon>0\\ &s.t.\ \forall x\in X, n\ge 0, \mu(B_{F_n}(x,\epsilon))\le C\exp(-S_{F_n}\psi(x))\}
\end{align*}
and
\begin{align*}
  \mathcal{V}^-=\{\psi\in C(X,\R): &\forall \epsilon_0>0, \exists 0<\epsilon<\epsilon_0 \text{ and } C=C(\epsilon)\\ &s.t.\ \forall x\in X, n\ge 0, \mu(B_{F_n}(x,\epsilon))\ge C\exp(-S_{F_n}\psi(x))\}.
\end{align*}
\begin{theorem}\label{th-2}
  Let $(X,G)$, $\{F_n\}$, $\mu$, $\varphi$ and $c$ be as in Theorem \ref{th-1}.
  Then for $\psi\in \mathcal{V}^+$,
    \begin{align*}
    \limsup_{n\rightarrow\infty}\frac{1}{|F_n|}\log \mu(\{x\in X: A_{F_n}\varphi(x)\ge c\})\le \sup\{h_{\nu}(X,G)-\int\psi \, d\nu\},
    \end{align*}
    where the supremum is taken over all $\nu\in M(X,G)$ with $\int\varphi \, d\nu\ge c$.
\end{theorem}
\begin{theorem}\label{th-3}
  Let $(X,G)$, $\{F_n\}$, $\mu$, $\varphi$ and $c$ be as in Theorem \ref{th-1} and in addition, assume that $(X,G)$ has weak specification property.
  Then for $\psi\in \mathcal{V}^-$,
    \begin{align*}
    \liminf_{n\rightarrow\infty}\frac{1}{|F_n|}\log \mu(\{x\in X: A_{F_n}\varphi(x)> c\})\ge \sup\{h_{\nu}(X,G)-\int\psi \, d\nu\},
    \end{align*}
    where the supremum is taken over all $\nu\in M(X,G)$ with $\int\varphi \, d\nu>c$.
\end{theorem}
We will give the definition of weak specification property in next section. We also remark that when $G=\Z$, the weak specification property here coincides with
the specification property in \cite{Y}.

The above theorems are the amenable group action version of Lai-Sang Young's classical results in \cite{L}. But for the proof, we need to handle the techniques for amenable group actions.
For the proof of Theorem \ref{th-1}, we need to prove an amenable group version of Katok's entropy formula \cite{K}. For the proof of Theorem \ref{th-2}, we need to employ the idea and technique of
the proof of variational principle for amenable pressure (see for example, \cite{O}). For the proof of Theorem \ref{th-3}, we need a lemma via the quasi tiling of amenable groups.
We also should remark here that the referee mentioned to us the recent advances on the tilings of amenable groups by Downarowicz etc \cite{DHZ}. It says that each amenable group can be tiled by finite many shapes. Compare with the quasi tiling theory, this result can provide a better and simplified understanding on the structure of amenable groups.

\section{Preliminary}
\subsection{Entropy}\

Let $(X,G,\mu)$ be a measure-theoretic $G-$action dynamical system where $G$ is a discrete countable amenable group and $\mu$ is a $G-$invariant Borel probability measure on $X$.
Let $\mathcal {P}$  and $\mathcal {Q}$ be two finite measurable partitions of $X$ and denote by $\mathcal{P}\vee \mathcal{Q}=\{P\cap Q: P\in\mathcal{Q}\text{ and }Q\in\mathcal{Q}\}$.
For a finite subset $F$ in $G$, we denote by $\mathcal{P}_F=\bigvee_{g\in F}g^{-1}\mathcal{P}$.
Then the classical measure-theoretical entropy of $\mathcal {P}$ is defined by
$$h_{\mu}(G,\mathcal{P})=\liminf_{n\rightarrow+\infty}\frac{1}{|F_n|}H(\mathcal{P}_{F_n}),$$
where $\{F_n\}$ is any F{\o}lner sequence in $G$ and the definition is independent of the specific F{\o}lner sequence $\{F_n\}$.
The {\it measure-theoretical entropy} of the system $(X,G,\mu)$, $h_{\mu}(X,G)$, is the supremum of $h_{\mu}(G,\mathcal{P})$ over $\mathcal {P}$.

Let's recall the classical Shannon-McMillan-Breiman theorem for amenable group actions.
\begin{theorem}[Shannon-McMillan-Breiman(SMB) theorem, \cite{L,W}]
Let $(X,G,\mu)$ be an ergodic $G-$system. For any tempered F{\o}lner sequence $\{F_n\}$ in $G$ with $\lim\limits_{n\rightarrow+\infty}\frac{|F_n|}{\log n}=\infty$
and any finite measurable partition $\mathcal{P}$ of $X$,
$$\lim_{n\rightarrow+\infty}-\frac{1}{|F_n|}\log \mu(\mathcal{P}_{F_n}(x))=h_{\mu}(G,\mathcal{P}),$$
for $\mu$ almost every $x\in X$, where $\mathcal {P}(x)$ denotes the atom in $\mathcal{P}$ that contains $x$.
\end{theorem}

Suggested by the SMB theorem, we may define the entropy at $x\in X$ with respect to $\mu\in M(X)$ for $\{F_n\}$ by
\begin{align*}
  h_{\mu}(\{F_n\},x)=\lim_{\epsilon\rightarrow 0}\limsup_{n\rightarrow+\infty}-\frac{1}{|F_n|}\log \mu(B_{F_n}(x,\epsilon)).
\end{align*}
In \cite{ZC}, the authors showed that $h_{\mu}(\{F_n\},x)=h_{\mu}(X,G)$ for $\mu$-a.e. $x\in X$ whence $\mu\in E(X,G)$ and
the F{\o}lner sequence $\{F_n\}$ is tempered and satisfies $\lim\limits_{n\rightarrow+\infty}\frac{|F_n|}{\log n}=\infty$.

For $\nu\in M(X)$, we define the $\nu-$relative entropy with respect to $\mu$ for F{\o}lner sequence $\{F_n\}$ by
\begin{align*}
  h_{\mu}(\{F_n\};\nu)=\nu-{\rm ess} \sup h_{\mu}(\{F_n\},x).
\end{align*}

Now we return to the topological case. Let $(X,G)$ be a compact metric $G-$action topological dynamical system and $G$ a discrete countable amenable group.
Let $\{F_n\}$ be a F{\o}lner sequence. Then the topological entropy of $(X,G)$ is defined in the following way.

Let $\mathcal {U}$ and $\mathcal{V}$ be two open covers of $X$ and denote by $\mathcal{U}\vee \mathcal{V}=\{U\cap V: U\in\mathcal{U}\text{ and }V\in\mathcal{V}\}$.
For a finite subset $F$ in $G$, we denote by $\mathcal{U}_{F}=\bigvee_{g\in F}g^{-1}\mathcal{U}$. Let $N(\mathcal{U})$ denote the number of sets in a finite subcover
 of $\mathcal{U}$ with smallest cardinality. Then the topological entropy of $\mathcal{U}$ is
$$h_{top}(G,\mathcal {U})=\lim_{n\rightarrow+\infty}\frac{1}{|F_n|}\log N\big(\mathcal{U}_{F_n}\big).$$
It is shown that $h_{top}(G,\mathcal {U})$ is not dependent on the choice of the F{\o}lner sequences $\{F_n\}$.
And the {\it topological entropy} of $(X,G)$ is
$$h_{top}(X,G)=\sup_{\mathcal U} h_{top}(G,\mathcal {U}),$$
where the supremum is taken over all the open covers of $X$.

\subsection{Weak specification}\

The following definition generalizes the traditional specification property to  general group actions(see \cite{CL}).

Let $(X,G)$ be a compact metric $G-$action topological dynamical system and $G$ a discrete countable group (need not be amenable).
$(X,G)$ has {\it weak specification} if for any $\epsilon>0$ there exists a nonempty finite subset $F$ of $G$
with the following property: for any finite collection $F_1, \cdots, F_m$ of finite subsets $G$ with
$$FF_i\cap F_j=\emptyset \text{ for } 1\le i,j\le m,i\neq j,$$
and for any collection of points $x_1,\cdots,x_m \in X$, there exists a point $y\in X$ such that
$$d(gx_i,gy)\le \epsilon \text{ for all } g\in F_i, 1\le i\le m.$$

\subsection{Quasi tiling}\

The quasi tiling theory for amenable groups, set up by Ornstein and Weiss(see \cite{OW}), is a useful tool and technique in the study of dynamical systems for amenable group actions.
Let $F(G)$ denote the collection of finite subsets of $G$. Let $A,K\subset F(G)$ and $\delta>0$, the set $A$ is {\it $(K, \delta)$-invariant} if
\begin{align*}
  \frac{|B(A,K)|}{|A|}<\delta,
\end{align*}
where $B(A,K)=\{g\in G: Kg\cap A\neq\emptyset \text { and } Kg\cap(G\setminus A)\neq\emptyset\}$ is the {\it $K$-boundary} of $A$.
Let $\{A_1,A_2,\cdots,A_k\}\subset F(G)$ and $\epsilon\in(0,1)$. Subsets
$A_1,A_2,\cdots,A_k$ are {\it $\epsilon$-disjoint} if there are $\{B_1,B_2,\cdots,B_k\}\subset F(G)$ such that
\begin{enumerate}
  \item $B_i\subset A_i$ and $\frac{|B_i|}{|A_i|}>1-\epsilon$ for each $i$,
  \item $B_i$'s are mutually disjoint.
\end{enumerate}
For $\alpha\in(0,1]$, we say that $\{A_1,A_2,\cdots,A_k\}$ {\it $\alpha$-covers} $A\in F(G)$ if
\begin{align*}
  \frac{|\cup_{i=1}^kA_i\cap A|}{|A|}\ge\alpha.
\end{align*}
We say that $A_1,A_2,\cdots,A_k$ {\it $\epsilon$-quasi-tile} $A\in F(G)$ if there exists $\{C_1,C_2,\cdots,C_k\}\subset F(G)$ such that
\begin{enumerate}
\item for each $i$, $A_iC_i\subset A$ and $A_ic$'s for $c\in C_i$ are $\epsilon$-disjoint,
\item for $i=1,2,\cdots,k$, $A_iC_i$'s are mutually disjoint,
\item $\{A_1C_1,A_2C_2,\cdots,A_kC_k\}$ $(1-\epsilon)$-covers $A$.
\end{enumerate}

The following theorem states a fundamental quasi tiling property of amenable groups.

\begin{theorem}[Theorem 2.4 of \cite{WZ}]
Let $G$ be an amenable group and $\{e\}\subset F_1\subset F_2\subset\cdots$ be a F{\o}lner sequence in $G$. Then for
any $0<\epsilon<\frac{1}{4}$ and any integer $N>0$, there exist integers $N\le n_1<n_2<\cdots<n_k$ such that any $F_M$ (M sufficiently large)
can be $\epsilon$-quasi-tiled by $F_{n_1},F_{n_2},\cdots, F_{n_k}$.
\end{theorem}

In our setting, we do not restrict the F{\o}lner sequence to be increasing and contain $e$. Let $\{F_n\}$ be a F{\o}lner sequence in $G$.
For any integer $N>0$, as in the proof of Theorem 2.4 of \cite{WZ}, choose $k>0$ and $\delta$ such that $(1-\frac{\epsilon}{2})^k<\epsilon$ and
$6^k\delta<\frac{\epsilon}{2}$. And choose $N\le n_1<n_2<\cdots<n_k$ such that $F_{n_{i+1}}$ is $(F_{n_i}F_{n_i}^{-1},\delta)$-invariant and $\frac{|F_{n_i}|}{|F_{n_{i+1}}|}<\delta$.
Then there exists some $g_i\in G$ such that $F_{n_i}g_i\subset F_{n_{i+1}}$ for each $1\le i\le k-1$. Take $g_k\in (g_1g_2\cdots g_{k-1})^{-1}F_{n_1}^{-1}$ and
set $\bar F_{n_i}=F_{n_i}g_ig_{i+1}\cdots g_k$ for each $1\le i\le k$. Then $\{e\}\subset \bar F{n_1}\subset \bar F_{n_2}\subset\cdots\subset\bar F_{n_k}$. Moreover,
$\bar F_{n_{i+1}}$ is $(\bar F_{n_i}\bar F_{n_i}^{-1},\delta)$-invariant and $\frac{|\bar F_{n_i}|}{|\bar F_{n_{i+1}}|}<\delta$. Then due to the proof of Theorem 2.4 of \cite{WZ},
any $(\bar F_{n_k}\bar F_{n_k}^{-1}, \delta)$-invariant finite subset $F$ with $\frac{|\bar F_{n_k}|}{|F|}<\delta$ can be $\epsilon$-quasi-tiled by
$\bar F_{n_1},\bar F_{n_2},\cdots, \bar F_{n_k}$. Since $\bar F_{n_i}$'s are translations of $F_{n_i}$'s, $F$ can also be $\epsilon$-quasi-tiled by
$F_{n_1},F_{n_2},\cdots, F_{n_k}$. Hence we get the following proposition.

\begin{proposition}\label{prop-2-3}
Let $G$ be an amenable group and $\{F_n\}$ be a F{\o}lner sequence in $G$. Then for
any $0<\epsilon<\frac{1}{4}$ and any integer $N>0$, there exist integers $N\le n_1<n_2<\cdots<n_k$ such that any $F_M$ (M sufficiently large)
can be $\epsilon$-quasi-tiled by $F_{n_1},F_{n_2},\cdots, F_{n_k}$.
\end{proposition}

\section{Katok's entropy formula and proof of Theorem \ref{th-1}}
In this section, we will prove Theorem \ref{th-1}. For this aim, we need the Katok's entropy formula \cite{K} for amenable group action dynamical systems.

\begin{theorem}\label{Katok}
Let $(X,G)$ be a compact metric $G-$action topological dynamical system and $G$ a discrete countable amenable group.
Let $\mu\in E(X,G$) and $\{F_n\}$ a tempered F{\o}lner sequence in $G$
with $\lim\limits_{n\rightarrow+\infty}\frac{|F_n|}{\log n}=\infty$, then
\begin{align*}
&\lim_{\epsilon\rightarrow 0}\liminf_{n\rightarrow+\infty}\frac{1}{|F_n|}\log N(F_n,\epsilon,\delta) \\
=&\lim_{\epsilon\rightarrow 0}\limsup_{n\rightarrow+\infty}\frac{1}{|F_n|}\log N(F_n,\epsilon,\delta)=h_{\mu}(X,G),
\end{align*}
where $N(F,\epsilon,\delta)$ is the minimal number of $B_{F}(x,\epsilon)$ balls that cover a set of measure no less than $1-\delta$.
\end{theorem}
\begin{proof}
We first show that
\begin{align}\label{ineq-1}
\limsup_{n\rightarrow+\infty}\frac{1}{|F_n|}\log N(F_n,\epsilon,\delta)\le h_{\mu}(X,G),
\end{align}
for any $\epsilon>0$ and $0<\delta<1$.
Take a finite measurable partition of $X$ which we denote by $\xi$ such that $diam(\xi)<\epsilon$. Then it follows that each atom in $\xi_{F_n}$
must be contained in some $B_{F_n}(x,\epsilon)$. Let
$$A_{F_n,\epsilon,\gamma}=\{x\in X: \mu(\xi_{F_n}(x))>\exp(-|F_n|(h_{\mu}(G,\xi)+\gamma))\}.$$
Since $\mu$ is ergodic, by SMB Theorem, for any $\gamma>0$, $\mu(A_{F_n,\epsilon,\gamma})\rightarrow 1$ as $n$ goes to infinity.
Then for sufficiently large $n$, $\mu(A_{F_n,\epsilon,\gamma})$ is bigger than $1-\delta$. The set $A_{F_n,\epsilon,\gamma}$ contains
at most $\lceil\exp(|F_n|(h_{\mu}(G,\xi)+\gamma))\rceil$ many atoms of $\xi_{F_n}$ and then can be covered by the corresponding $B_{F_n}(x,\epsilon)$ balls.
This deduces that for any $\gamma>0$,
$$\limsup_{n\rightarrow+\infty}\frac{1}{|F_n|}\log N(F_n,\epsilon,\delta)\le h_{\mu}(G,\xi)+\gamma.$$
Letting $\gamma\rightarrow 0$, we get the inequality \eqref{ineq-1}.

In the following we will show that
\begin{align}\label{ineq-2}
\lim_{\epsilon\rightarrow 0}\liminf_{n\rightarrow+\infty}\frac{1}{|F_n|}\log N(F_n,\epsilon,\delta)\ge h_{\mu}(X,G).
\end{align}

Let $\xi$ be a finite measurable partition of $X$ such that $\mu(\partial \xi)=0$ and denote by $\tilde{h}=h_{\mu}(G,\xi)$. Then for
sufficiently small $\gamma>0$, the $\gamma-$neiborghhood of $\partial \xi$ (denoted by $U_{\gamma}$) has measure less than $\epsilon$.
By the ergodic theorem, $\frac{1}{|F_n|}\sum\limits_{g\in F_n}\chi_{U_{\gamma}}(gx)$ converges to $\mu(U_{\gamma})$ a.e..
For any $\epsilon>0$, notice that
\begin{align*}
&\{x\in X: \lim_{n\rightarrow+\infty}\frac{1}{|F_n|}\sum\limits_{g\in F_n}\chi_{U_{\gamma}}(gx)=\mu(U_{\gamma})\}\\
\subseteq &\{x\in X: \lim_{n\rightarrow+\infty}\frac{1}{|F_n|}\sum\limits_{g\in F_n}\chi_{U_{\gamma}}(gx)<\epsilon\}\\
\subseteq &\bigcup_{N\in \N}\{x\in X: \forall n>N, \frac{1}{|F_n|}\sum\limits_{g\in F_n}\chi_{U_{\gamma}}(gx)<\epsilon\}\\
\triangleq & \bigcup_{N\in \N} E_N
\end{align*}
 and $E_N$ increases. For sufficiently large $N$, whence $n>N$, we have
\begin{align}\label{set-1}
 \mu(\{x\in X: \forall n'\ge n, \sum\limits_{g\in F_{n'}}\chi_{U_{\gamma}}(gx)<\epsilon|F_{n'}|\})>1-\epsilon.
\end{align}

By SMB Theorem, $-\frac{1}{|F_n|}\log \mu(\xi_{F_n}(x))$ converges to $\tilde{h}$ a.e.. Hence by the same argument as above,
for sufficiently large $N$, whence $n>N$,
\begin{align}\label{set-2}
 \mu(\{x\in X: \forall n'\ge n, -\frac{1}{|F_{n'}|}\log \mu(\xi_{F_{n'}}(x))>\tilde{h}-\epsilon\})>1-\epsilon.
\end{align}

Let \begin{align}\label{set-3}
      E=&\{x\in X: \forall n'\ge n, \sum\limits_{g\in F_{n'}}\chi_{U_{\gamma}}(gx)<\epsilon|F_{n'}|\}\\
      &\bigcap\{x\in X: \forall n'\ge n,-\frac{1}{|F_{n'}|}\log \mu(\xi_{F_{n'}}(x))> \tilde{h}-\epsilon\}. \nonumber
    \end{align}
 Then for any $n>N$, $\mu(E)>1-2\epsilon$.

Let $w_{\xi,F_n}(x)=(\xi(gx))_{g\in F_{n}}$ be the $(\xi,F_n)-$name of $x$. For any $y\in B(x,\gamma)$, we have that either $\xi(x)=\xi(y)$ or $x\in U_{\gamma}(\xi)$.
Hence if $x\in E$ and $y\in B_{F_n}(x,\gamma)$, then the Hamming distance between $w_{\xi,F_n}(x)$ and $w_{\xi,F_n}(y)$ is less than $\epsilon$. This implies that
whence $x\in E$, $$B_{F_n}(x,\gamma)\subseteq \bigcup \{\xi_{F_n}(y): w_{\xi,F_n}(y) \text{ is }\epsilon-\text{close to }w_{\xi,F_n}(x)\text{ under Hamming metric}\}.$$

Denote the total number of such $(\xi,F_n)-$names by $L_n$, then when $n$ is large enough,
\begin{align*}
  L_n\le\sum_{j=0}^{\lfloor\epsilon|F_n|\rfloor}\binom{|F_n|}{j}(\#\xi-1)^j\le \epsilon|F_n|\binom{|F_n|}{\lfloor\epsilon|F_n|\rfloor}(\#\xi-1)^{\epsilon|F_n|}.
\end{align*}
For an upper bound of $L_n$, one may refer to \cite{K} or \cite{BK}. Here we will give a brief estimate.
Applying the Stirling formula
\begin{align*}
  n!=\sqrt{2\pi n}(\frac{n}{e})^ne^{\alpha_n}, \frac{1}{12n+1}<\alpha_n<\frac{1}{12n},
\end{align*}
$L_n$ can be estimated by:
\begin{align*}
 L_n&\le\epsilon|F_n|(\#\xi-1)^{\epsilon|F_n|}\cdot\frac{\sqrt{2\pi |F_n|}e^{\alpha_{|F_n|}-\alpha_{|F_n|-\lfloor\epsilon|F_n|\rfloor}-\alpha_{\lfloor\epsilon|F_n|\rfloor}}}{\sqrt{2\pi (|F_n|-\lfloor\epsilon|F_n|\rfloor)}\sqrt{2\pi \lfloor\epsilon|F_n|\rfloor}}\cdot\\ &\ \ \ \ \ \ \ \ (1-\frac{\lfloor\epsilon|F_n|\rfloor}{|F_n|})^{-(|F_n|-\lfloor\epsilon|F_n|\rfloor)}(\frac{\lfloor\epsilon|F_n|\rfloor}{|F_n|})^{-\lfloor\epsilon|F_n|\rfloor}\\
 &\le \exp(\epsilon|F_n|)\cdot \exp(\epsilon|F_n|\log(\#\xi-1))\cdot \exp(-\epsilon|F_n|\log\epsilon-(1-\epsilon)|F_n|\log(1-\epsilon))\\
 &=\exp(\eta|F_n|),
\end{align*}
where
$$\eta=\epsilon+\epsilon\log(\#\xi-1)-\epsilon\log\epsilon-(1-\epsilon)\log(1-\epsilon).$$
We note that $\eta>0$ is a constant only dependent on $\#\xi$ and $\epsilon$ but independent of $x$ and $n$ and moreover, $\eta$ tends to $0$ as $\epsilon$ tends to $0$.

Let $D$ be a subset of $X$ with measure no less than $1-\delta$. If we let $\epsilon<\frac{1-\delta}{4}$, then $\mu(D\cap E)>\frac{1-\delta}{2}$.
We will estimate a lower bound for the number of $B_{F_n}(x,\gamma)$ balls needed to cover the set $D\cap E$. Suppose $D\cap E$ is covered by the family of
balls $\{B_{F_n}(y,\gamma): y\in S\}$.
From the discussion above, the intersection of $D\cap E$ with each ball in this family can be covered by no more than $L_n$ many atoms of $\xi_{F_n}$ and any atom of $\xi_{F_n}$ that intersects $D\cap E$ has measure at most $\exp(-|F_n|(\tilde{h}-\epsilon))$. Hence
$$\#S>\frac{\frac{1-\delta}{4}}{L_n\cdot\exp(-|F_n|(\tilde{h}-\epsilon))}\ge\exp(|F_n|(\tilde{h}-\eta-\epsilon))\cdot \frac{1-\delta}{4},$$
which allows us to deduce that
$$\liminf_{n\rightarrow+\infty}\frac{1}{|F_n|}\log N(F_n,\gamma,\delta)\ge \tilde{h}-\eta-\epsilon.$$
Letting $\epsilon$ goes to $0$ and then taking a sequence of $\xi$'s whose diameters tend to $0$, we obtain \eqref{ineq-2}.

\end{proof}

\begin{proof}[Proof of Theorem \ref{th-1}] \

  Let $\nu\in E(X,G)$ with $\int\varphi \, d\nu>c$. Since $h_{\nu}(X,G)\le h_{top}(X,G)<\infty$, the theorem holds obviously for the case that $h_{\mu}(G;\nu)=\infty$.
  Now we treat for the case $h_{\mu}(G;\nu)<\infty$.

  Denote by $V_{n}=\{x\in X: A_{F_n}\varphi(x)>c\}$ and let $\delta=\frac{1}{2}(\int\varphi \, d\nu-c)$.
  Choose $\epsilon_0>0$ such that $|\varphi(x)-\varphi(y)|<\delta$, whenever $d(x,y)<\epsilon_0$. Hence if $A_{F_n}\varphi(x)>\int\varphi \, d\nu-\delta=c+\delta$, then
  $B_{F_n}(x,\epsilon_0)\subseteq V_{n}$.

  For any $\gamma>0$, by Theorem \ref{Katok}, we can choose $0<\epsilon_1\le\epsilon_0$ such that for any $0<\epsilon<\epsilon_1$,
  \begin{align*}
    \liminf_{n\rightarrow\infty}\frac{1}{|F_n|}\log N(F_n,4\epsilon, \frac{1}{2})\ge h_{\nu}(X,G)-\gamma.
  \end{align*}
  And by the definition of $h_{\mu}(G;\nu)$, we can choose $0<\epsilon_2\le\epsilon_1$ such that
  \begin{align*}
    \nu(\{x\in X:\limsup_{n\rightarrow\infty}-\frac{1}{|F_n|}\log \mu(B_{F_n}(x,\epsilon_2))\le h_{\mu}(G;\nu)+\gamma\})>1-\frac{1}{10}.
  \end{align*}
  Denote by $E=\{x\in X:\limsup_{n\rightarrow\infty}-\frac{1}{|F_n|}\log \mu(B_{F_n}(x,\epsilon_2))\le h_{\mu}(G;\nu)+\gamma\}$. Then there exists $N\in\N$,
  such that $$\nu(\{x\in E: \mu(B_{F_n}(x,\epsilon_2))\ge \exp(-(h_{\mu}(G;\nu)+2\gamma)|F_n|), \forall n\ge N\})>1-\frac{1}{5}$$
  and $$\nu(\{x\in X: A_{F_n}\varphi(x)>\int\varphi \, d\nu-\delta, \forall n\ge N\}>1-\frac{1}{5}.$$
  Let $\Gamma$ be the intersection of the above two sets. Then $\nu(\Gamma)>1-\frac{2}{5}>\frac{1}{2}$.

  For each $n\ge N$, let $E_n$ be a maximal set of $(F_n, 2\epsilon_2)-$separated points in $\Gamma$, i.e. for any two different points $x,y\in E$, $d_{F_n}(x,y)\ge 2\epsilon_2$.
  By the maximality of $E_n$, $\Gamma\subseteq\bigcup_{x\in E_n}B_{F_n}(x,4\epsilon_2)$. Hence $\#E_n\ge N(F_n,4\epsilon_2,\frac{1}{2})$.
  Moreover, the Bowen balls $B_{F_n}(x,\epsilon_2)$, $x\in E_n$, are mutually disjoint. Noticing that for $x\in E_n\subseteq \Gamma$,
  $B_{F_n}(x,\epsilon_2)\subseteq V_{n}$, we have
  $$\mu(V_{n})\ge \sum_{x\in E_n}\mu(B_{F_n}(x,\epsilon_2))\ge \#E_n\exp(-(h_{\mu}(G;\nu)+2\gamma)|F_n|).$$
  Hence
  $$\liminf_{n\rightarrow\infty}\frac{1}{|F_n|}\log \mu(V_{n})\ge \liminf_{n\rightarrow\infty}\frac{1}{|F_n|}\log N(F_n,4\epsilon_2,\frac{1}{2})-(h_{\mu}(G;\nu)+2\gamma.$$
  Letting $\gamma$ and $\epsilon_2$ tend to $0$, we finish the proof of Theorem \ref{th-1}.
\end{proof}

\section{Proof of Theorem \ref{th-2}}

\begin{proof}[Proof of Theorem \ref{th-2}]\

  Let $\psi\in\mathcal{V}^+$ and denote $V_n=\{x\in X: A_{F_n}\varphi(x)\ge c\}$. We will show that there exists $\nu\in M(X,G)$ with $\int\varphi \, d\nu\ge c$
  such that
    \begin{align*}
      \limsup_{n\rightarrow\infty}\frac{1}{|F_n|}\log \mu(V_n)\le h_{\nu}(X,G)-\int\psi \, d\nu.
    \end{align*}
    By the definition of $\mathcal{V}^+$, we may let $C$ and $\epsilon$ be such that for any $x\in X$ and $n\in\N$, $\mu(B_{F_n}(x,\epsilon))\le C\exp(-S_{F_n}\psi(x))$.
    For each $n$, let $E_n$ be a maximal $(F_n,\epsilon)$-separated set contained in $V_n$.

    Let $\sigma_{n}\in M(X)$ be the atomic measure concentrated on $E_n$ by the formula
    \begin{align*}
      \sigma_{n}=\frac{\sum_{x\in E_n}e^{-S_{F_n}\psi(x)}\delta_x}{\sum_{z\in E_n}e^{-S_{F_n}\psi(z)}}.
    \end{align*}
    Define $\nu_n\in M(X)$ by
    \begin{align*}
      \nu_{n}=\frac{1}{|F_n|}\sum_{g\in F_n}\sigma_n\circ g^{-1}
    \end{align*}
    We note here that $\int S_{F_n}\psi \, d\sigma_n=\int\psi \, d\nu_n$.

    For convenience we denote $Z_n=\sum_{z\in E_n}e^{-S_{F_n}\psi(z)}$. Assume that the upper limit of $\frac{1}{|F_n|}\log Z_n$
    is attained along a sequence $\{n_j\}$, i.e.
    $$\limsup_{n\rightarrow\infty}\frac{1}{|F_n|}\log Z_n=\lim_{j\rightarrow\infty}\frac{1}{|F_{n_j}|}\log Z_{n_j}.$$
    Let $\nu$ be a weak limit of $\nu_{n_j}$. We may assume $\nu_{n_j}\rightarrow\nu$ by taking a subsequence and still denote this subsequence by $\{n_j\}$.
    Then $\nu\in M(X,G)$.

    Also notice that
    \begin{align*}
      \int \psi \, d\nu_n&=\int \psi \, d\frac{1}{|F_n|}\sum_{g\in F_n}\sigma_n\circ g^{-1}\\
      &=\int \psi \, d\frac{1}{Z_n}\sum_{x\in E_n}e^{-S_{F_n}\psi(x)}\frac{1}{|F_n|}\sum_{g\in F_n}\delta_{gx}\\
      &=\frac{1}{Z_n}\sum_{x\in E_n}e^{-S_{F_n}\psi(x)}A_{F_n}\psi(x)\ge c,
    \end{align*}
    we have that $\int\psi \, d\nu\ge c$.

    Let $\beta$ be a finite Borel partition of $X$ each element of which has diameter less than $\epsilon$. Moreover, we require that $\nu(\partial \beta)=0$.
    Recall that for any finite subset $F$ of $G$, $\beta_{F}=\bigvee_{g\in F}g^{-1}\beta$.
    Then each element of $\beta_{F_n}$ contains at most one point in $E_n$. Hence
    \begin{align*}
      &H_{\sigma_n}(\beta_{F_n})-\int S_{F_n}\psi \, d\sigma_n\\
      &=\sum_{y\in E_n}-\sigma_n(\{y\})\log \sigma_n(\{y\})-\sum_{y\in E_n}S_{F_n}\psi(y)\sigma_n(\{y\})\\
      &=\sum_{y\in E_n}\sigma_n(\{y\})(-S_{F_n}\psi(y)-\log \sigma_n(\{y\}))\\
      &=\sum_{y\in E_n}\sigma_n(\{y\})(-S_{F_n}\psi(y)-\log \frac{e^{-S_{F_n}\psi(y)}}{Z_n})\\
      &=\log Z_n.
    \end{align*}
    By Lemma 3.1 (3) of \cite{HYZ}, the multi-subadditivity of $H_{\sigma_n}(\beta_\bullet)$, we have:
    \begin{align*}
      H_{\sigma_n}(\beta_{F_n})
      &\le \frac{1}{|F|}\sum_{g\in F_n}H_{\sigma_n\circ g^{-1}}(\beta_{F})+|F^{-1}F_n\setminus F_n|\log\#\beta.
    \end{align*}
    Hence
    \begin{align*}
      &\;\;\;\;\frac{1}{|F_n|}H_{\sigma_n}(\beta_{F_n})\\
      &\le \frac{1}{|F|}\frac{1}{|F_n|}\sum_{g\in F_n}H_{\sigma_n\circ g^{-1}}(\beta_{F})+\frac{1}{|F_n|}|F^{-1}F_n\setminus F_n|\log\#\beta\\
      &\le \frac{1}{|F|}H_{\nu_n}(\beta_F)+\frac{|F^{-1}F_n\setminus F_n|}{|F_n|}\log\#\beta.
    \end{align*}
    Because $\nu(\partial \beta_F)=0$, we have that
    \begin{align*}
      \limsup_{n\rightarrow\infty}\frac{1}{|F_n|}\log Z_n&=\lim_{j\rightarrow\infty}\frac{1}{|F_{n_j}|}\log Z_{n_j}\\
      &=\lim_{j\rightarrow\infty} \frac{1}{|F_{n_j}|}\big(H_{\sigma_{n_j}}(\beta_{F_{n_j}})-\int S_{F_{n_j}}\psi \, d\sigma_{n_j}\big)\\
      &\le \lim_{j\rightarrow\infty}\big(\frac{1}{|F|}H_{\nu_{n_j}}(\beta_F)+\frac{|F^{-1}F_{n_j}\setminus F_{n_j}|}{|F_{n_j}|}\log\#\beta\big )
      -\lim_{j\rightarrow\infty}\int\psi \, d\nu_{n_j}\\
      &= \frac{1}{|F|}H_{\nu}(\beta_F)-\int\psi \, d\nu,
    \end{align*}
    which allows us to deduce that
    \begin{align*}
      \limsup_{n\rightarrow\infty}\frac{1}{|F_n|}\log Z_n\le h_{\nu}(X,G)-\int\psi \, d\nu
    \end{align*}
    by taking $F$ over all the finite subsets of $G$.

    Since $V_n\subseteq \bigcup_{x\in E_n}B_{F_n}(x,\epsilon)$, we have
    \begin{align*}
     \limsup_{n\rightarrow\infty}\frac{1}{|F_n|}\log \mu(V_n)
     &\le \limsup_{n\rightarrow\infty}\frac{1}{|F_n|}\log \sum_{x\in E_n}\mu(B_{F_n}(x,\epsilon))\\
     &\le \limsup_{n\rightarrow\infty}\frac{1}{|F_n|}\log \sum_{x\in E_n}C\exp(-S_{F_n}\psi(x))\\
     &= \limsup_{n\rightarrow\infty}\frac{1}{|F_n|}\log Z_n\\
     &\le h_{\nu}(X,G)-\int\psi \, d\nu.
    \end{align*}

\end{proof}

\section{Proof of Theorem \ref{th-3}}
\begin{proof}[Proof of Theorem \ref{th-3}]\

  Let $\psi\in\mathcal{V}^-$ and $V_n=\{x\in X: A_{F_n}\varphi(x)> c\}$.
  For any $\nu\in M(X,G)$ with $\int \phi \, d\nu>c$,
  let $\delta=\frac{1}{7}(\int \phi \, d\nu-c)$ and let $0<\gamma< \frac{\delta}{12}$. Then pick $\epsilon_0>0$ such that for any $\tau_1,\tau_2\in M(X,G)$, it holds that
  $$|\int \phi \, d\tau_1-\int\phi \, d\tau_2|<\delta \text{ and }|\int \psi \, d\tau_1-\int\psi \, d\tau_2|<\gamma,$$
  whenever $d_{M(X)}(\tau_1,\tau_2)\le \epsilon_0$, where $d_{M(X)}$ is a metric compatible with the weak topology of $M(X)$.

  Now let $\alpha=\{A_1,A_2,\cdots,A_k\}$ be a partition of $M(X,G)$ with $diam(\alpha)\le\epsilon_0$.
  Let $\nu=\int_{E(X,G)}\tau \, d\pi(\tau)$ be the ergodic decomposition of $\nu$, where $\pi$ is the corresponding Borel probability measure on $E(X,G)$.
  Hence $\int f \, d\nu=\int_{E(X,G)}(\int f \, d\tau)\, d\pi(\tau)$ for any $f\in C(X,\R)$. Let $a_i=\pi(A_i)$ and choose $\lambda_i\in A_i\cap E(X,G)$
  such that $h_{\lambda_i}(X,G)\ge h_{\tau}(X,G)-\gamma$ for $\pi-$a.e. $\tau\in A_i$, $i=1,2,\cdots,k$.

  Let $\lambda=\sum_{i=1}^ka_i\lambda_i$. Then $d_{M(X)}(\lambda,\nu)\le \epsilon_0$. Hence
  $$\int \phi \, d\lambda>\int \phi \, d\nu -\delta=c+6\delta$$ and
  $$h_{\lambda}(X,G)-\int \psi d \lambda\ge h_{\nu}(X,G)-\int \psi \, d\nu-2\gamma.$$
  Denote by $M=\max\{\sup_{x\in X}|\psi(x)|,\sup_{x\in X}|\phi(x)|,1\}$.
  Choose $\epsilon>0$ sufficiently small such that the following condition holds.
  \begin{enumerate}
    \item[{\bf (C1)}] If $d(x,y)<\epsilon$, then $$|\phi(x)-\phi(y)|<\delta \text{ and }|\psi(x)-\psi(y)|<\gamma.$$
  \end{enumerate}
  Let $L$ be the maximal cardinality of an $\epsilon$-separated set in $X$. And then choose $N\in\N$ sufficiently large such that the following three conditions hold.
  \begin{enumerate}
    \item[{\bf (C2)}] For all $n\ge N$, $\frac{|F_n\vartriangle gF_n|}{|F_n|}<\frac{\gamma}{kML|F|}$, where $F$ is the finite subset of $G$ associated with $\gamma$ in the definition of weak specification of $(X,G)$.
    \item[{\bf (C3)}] For each $1\le i\le k$, there exists $E_i\subseteq X$ with $\lambda_i(E_i)>\frac{1}{2}$ such that for every $x\in E_i$ and $n\ge N$,
    \begin{enumerate}
      \item $S_{F_n}\psi(x)\le|F_n|(\int\psi \, d\lambda_i+\gamma)$,
      \item $S_{F_n}\phi(x)\ge|F_n|(\int\phi \, d\lambda_i-2\delta)$.
    \end{enumerate}
    \item[{\bf (C4)}] $\forall n\ge N$, $N(F_n,4\epsilon,\frac{1}{2})>\exp(|F_n|(h_{i}-\gamma))$, for each $1\le i\le k$, where $h_{i}=h_{\lambda_i}(X,G)$.
  \end{enumerate}
  We note that {\bf (C2)} can be satisfied since $\{F_n\}$ is a F{\o}lner sequence, {\bf (C3)} holds by the ergodic theorem and {\bf (C4)} holds by Theorem \ref{Katok}.

  By Proposition \ref{prop-2-3}, there exists $N_1>N$ large enough, such that for any $n\ge N_1$, $F_{n}$ can be $\frac{\gamma}{kML|F|}$-quasi tiled by $\{F_{n_1},F_{n_2},\cdots,F_{n_l}\}$ with
  tiling centers $\{C_1,C_2,\cdots,C_l\}$ and all $n_j$'s are no less than $N$. We may require $N_1$ large enough to ensure that the family of translations
  $\mathcal{F}=\{F_{n_j}c_j: 1\le j\le l,c_j\in C_j\}$ can be partitioned into $k$ subfamilies $\mathcal{F}_i, 1\le i\le k$, such that
  \begin{align}\label{a_i-weighted}
    |\frac{|\cup \mathcal{F}_i|}{|F_n|}-a_i|<\frac{3\gamma}{kML|F|}.
  \end{align}
  Since the family $\mathcal{F}$ is $\frac{\gamma}{ML|F|}$-disjoint, we have
  \begin{align*}
    (1-\frac{\gamma}{ML|F|})\sum_{F_{n_j}c_j\in \mathcal{F}}|F_{n_j}|\le |\cup \mathcal{F}|\le |F_n|.
  \end{align*}
  Hence
  \begin{align}\label{F-weighted00}
    \sum_{F_{n_j}c_j\in \mathcal{F}}|F_{n_j}|-|\cup \mathcal{F}|\le \frac{\frac{\gamma}{ML|F|}}{1-\frac{\gamma}{ML|F|}}|F_n|\le \frac{2\gamma}{M}|F_n|
  \end{align}
  and
  \begin{align}\label{F-weighted0}
    \sum_{F_{n_j}c_j\in \mathcal{F}}|F_{n_j}|\le \frac{4}{3}|F_n|,
  \end{align}
  whence $\gamma$ is sufficiently small.

  {\bf Claim.} For each $F_{n_j}c_j$ there exists a subset $T_{c_j}$ of $F_{n_j}$ such that
  \begin{enumerate}
    \item $T_{c_j}c_j$'s are pairwise disjoint,
    \item $FT_{c_j}c_j$'s are also pairwise disjoint,
    \item $\frac{|T_{c_j}|}{|F_{n_j}|}>1-\frac{3\gamma}{ML}$.
  \end{enumerate}
  \begin{proof}[\bf Proof of the Claim.]
  To obtain $T_{c_j}$'s, we first notice that since the translations in $\mathcal{F}$ are $\frac{\gamma}{|F|}$-disjoint, there must exists for each $F_{n_j}c_j$ a subset
  $\tilde {T}_{c_j}\subseteq F_{n_j}$ with $\frac{|\tilde{T}_{c_j}|}{|F_{n_j}|}>1-\frac{\gamma}{ML|F|}$ such that $\tilde {T}_{c_j}c_j$'s are pairwise disjoint.
  We then let $T_{c_j}=\{t\in F_{n_j}: Ft\subseteq \tilde{T}_{c_j}\}\cap \tilde{T}_{c_j}$. Obviously $T_{c_j}\subseteq F_{n_j}$ and $T_{c_j}c_j$'s are pairwise disjoint.
  Moreover, $FT_{c_j}\subseteq \tilde{T}_{c_j}$ implies that $FT_{c_j}c_j$'s are also pairwise disjoint.
  From the construction of $T_{c_j}$,
  \begin{align*}
    F_{n_j}\setminus T_{c_j}&\subseteq \{t\in F_{n_j}: \exists g\in F \text{ such that }gt\notin \tilde{T}_{c_j}\}\cup (F_{n_j}\setminus \tilde{T}_{c_j})\\
    &\subseteq \cup_{g\in F}\{t\in F_{n_j}: gt\notin \tilde{T}_{c_j}\}\cup (F_{n_j}\setminus \tilde{T}_{c_j})\\
    &\subseteq \cup_{g\in F}\big (\{t\in F_{n_j}: gt\notin F_{n_j}\}\cup\{t\in F_{n_j}: gt\in F_{n_j}\setminus\tilde{T}_{c_j} \}\big )\cup (F_{n_j}\setminus \tilde{T}_{c_j}).
  \end{align*}
  Thus
  \begin{align*}
    |F_{n_j}\setminus T_{c_j}|&\le \sum_{g\in F}\big (|\{t\in F_{n_j}: gt\notin F_{n_j}\}|+|\{t\in F_{n_j}: gt\in F_{n_j}\setminus\tilde{T}_{c_j} \}|\big )+|F_{n_j}\setminus \tilde{T}_{c_j}|\\
    &\le \sum_{g\in F}\big (|gF_{n_j}\setminus F_{n_j}|+|F_{n_j}\setminus \tilde{T}_{c_j}|\big )+|F_{n_j}\setminus \tilde{T}_{c_j}|\\
    &\le(2|F|+1)|F_{n_j}|\frac{\gamma}{ML|F|}\le \frac{3\gamma}{ML}|F_{n_j}|,
  \end{align*}
  which implies that $\frac{|T_{c_j}|}{|F_{n_j}|}>1-\frac{3\gamma}{ML}$. This finishes the proof of the claim.
  \end{proof}

  We should mention here that $C_j$'s may not be pairwise disjoint, hence $T_{c_j}$ may not be associated with a unique index $c_j$.
  But if we replace the index $c_j$ by a pair $(c_j,j)$, this situation can be avoided.

  Denote by $\tilde{\mathcal{F}}=\{T_{c_j}c_j: F_{n_j}c_j\in \mathcal{F}\}$ and $\tilde{\mathcal{F}}_i=\{T_{c_j}c_j: F_{n_j}c_j\in \mathcal{F}_i\}$.
  Then
  \begin{align}\label{F-weighted}
    |\cup \tilde{\mathcal{F}}|=\sum_{F_{n_j}c_j\in {\mathcal{F}}}|T_{n_j}c_j|\ge (1-\frac{3\gamma}{ML})\sum_{F_{n_j}c_j\in {\mathcal{F}}}|F_{n_j}c_j|\ge (1-\frac{3\gamma}{ML})|F_n|.
  \end{align}

  Since $n_j\ge N$, for each $1\le i\le k$ and $F_{n_j}$, by {\bf (C3)}, for each $D=T_{c_j}c_j\in \tilde{\mathcal{F}}_i$, there exists $E_D\subseteq X$
  with $\lambda_i(E_D)>\frac{1}{2}$ such that
 for each $x\in E_D$,
 \begin{align*}
      S_{T_{c_j}}\psi(x)&=S_{F_{n_j}}\psi(x)-S_{F_{n_j}\setminus T_{c_j}}\psi(x)\\
                        &\le|F_{n_j}|(\int\psi \, d\lambda_i+\gamma)+|F_{n_j}|\cdot (1-\frac{|T_{c_j}|}{|F_{n_j}|})M\\
                        &\le |F_{n_j}|(\int\psi \, d\lambda_i+4\gamma)
 \end{align*}
and
 \begin{align*}
    S_{T_{c_j}}\phi(x)&=S_{F_{n_j}}\phi(x)-S_{F_{n_j}\setminus T_{c_j}}\phi(x)\\
                      &\ge|F_{n_j}|(\int\phi \, d\lambda_i-2\delta)-|F_{n_j}|\cdot (1-\frac{|T_{c_j}|}{|F_{n_j}|})M\\
                      &\ge|F_{n_j}|(\int\phi \, d\lambda_i-2\delta-3\gamma).
 \end{align*}

 Since $\lambda_i(E_D)>\frac{1}{2}$, by {\bf (C4)}, there exist at least $\lceil\exp(|F_{n_j}|(h_{\lambda_i}-\gamma))\rceil$-many $(F_{n_j},4\epsilon)$-separated points in $E_D$.
 And hence at least $\frac{\lceil\exp(|F_{n_j}|(h_{\lambda_i}-\gamma))\rceil}{L^{|F_{n_j}\setminus T_{n_j}|}}$-many $(T_{n_j},4\epsilon)$-separated points in $E_D$.
 Denote the collection of these separated points by $\tilde{E}_D$.
 Note that for any two different elements in $\tilde{\mathcal{F}}$, say $D_1$ and $D_2$, we have that $FD_1\cap D_2=\emptyset$. We then can apply the weak specification property.
 To each tuple $(x_{D})_{D\in\tilde{\mathcal{F}}}\in X^{\tilde{\mathcal{F}}}$ with $x_{D}\in c_j^{-1}E_{i,j}$ for $D=T_{c_j}c_j\subset F_{n_j}c_j$,
 there is a shadowing point $y\in X$ such that $d(gy,gx_D)<\epsilon$, for any $g\in D$. Then $B_{\cup \tilde{F}}(y,\epsilon)$'s are disjoint for different $(x_{D})$'s.
 This can be seen by the following. Let $(x_D^1)$ and $(x_D^2)$ be any two different such tuples with $x^1_D\neq x^2_D$ for some $D\in\tilde{\mathcal{F}}$ and let $y^1$ and $y^2$
 be the corresponding shadowing points. Since $x^1_D\in B_{D}(y^1,\epsilon)$ and $x^1_D\in B_{D}(y^1,\epsilon)$, if $B_{D}(y^1,\epsilon)\cap B_{D}(y^2,\epsilon)\neq\emptyset$,
 then $$d_{T_{n_j}}(c_jx^1_D,c_jx^2_D)=d_{D}(x^1_D,x^2_D)<4\epsilon.$$ This contradicts with the fact that $c_jx^1_D\neq c_jx^2_D\in \tilde{E}_D$. Note that $\cup\tilde{\mathcal{F}}\subset F_n$, hence $B_{F_n}(y,\epsilon)$'s are also disjoint for different $(x_{D})$'s.

 Now let us give a lower bound for the total number $Q$ of different $(x_{D})$'s. Since for each $D=T_{c_j}c_j\in \tilde{\mathcal{F}}_i$, the total number of choices of $x_D$, denoted by $Q_D$, is at least
 $\frac{\lceil\exp(|F_{n_j}|(h_{\lambda_i}-\gamma))\rceil}{L^{|F_{n_j}\setminus T_{n_j}|}}$, we have
 \begin{align*}
   Q&=\Pi_{i=1}^k\Pi_{D=T_{c_j}c_j\in \tilde{\mathcal{F}}_i}Q_D\\
    &\ge \Pi_{i=1}^k\Pi_{D=T_{c_j}c_j\in \tilde{\mathcal{F}}_i}\frac{\exp(|F_{n_j}|(h_{\lambda_i}-\gamma))}{L^{|F_{n_j}\setminus T_{n_j}|}}\\
    &\ge \exp(\sum_{i=1}^k\sum_{F_{n_j}c_j\in {\mathcal{F}}_i}|F_{n_j}|(h_{\lambda_i}-\gamma-\frac{3\gamma}{ML}\cdot\log L))\\
    &\ge \exp\bigg(\sum_{i=1}^k|\cup\mathcal{F}_i|(h_{\lambda_i}-4\gamma)\bigg)\\
    &\ge \exp\bigg(|F_n|\sum_{i=1}^k(a_i-\frac{3\gamma}{kML|F|})(h_{\lambda_i}-4\gamma)\bigg) \text{ (by \eqref{a_i-weighted})}.
 \end{align*}

 For any $z\in B_{F_n}(y,\epsilon)$,
 \begin{align*}
   S_{F_n}\phi(z)&=S_{\cup\tilde{\mathcal{F}}}\phi(z)+S_{F_n\setminus\cup\tilde{\mathcal{F}}}\phi(z)\\
                 &\ge \sum_{D\in \tilde{F}}\big(S_D\phi(x)-|S_D\phi(x)-S_D\phi(y)|-|S_D\phi(y)-S_D\phi(z)|\big)\\
                 &\qquad-|F_n|\cdot M\cdot(1-\frac{|\cup\tilde{\mathcal{F}}|}{|F_n|})\\
                 &\ge \sum_{D\in \tilde{F}}S_D\phi(x)-|F_n|2\delta-|F_n|3\gamma\text{ (by \eqref{F-weighted})}\\
                 &= \sum_{i=1}^k\sum_{D=T_{n_j}c_j\in \tilde{\mathcal{F}}_i}S_{T_{n_j}}\phi(c_jx)-|F_n|(2\delta+3\gamma)\\
                 &\ge \sum_{i=1}^k\sum_{T_{n_j}c_j\in \tilde{\mathcal{F}}_i}|F_{n_j}|(\int\phi \, d\lambda_i-2\delta-3\gamma)-|F_n|(2\delta+3\gamma)\\
                 &= \sum_{i=1}^k\sum_{F_{n_j}c_j\in {\mathcal{F}}_i}|F_{n_j}|\int\phi \, d\lambda_i-\sum_{F_{n_j}c_j\in {\mathcal{F}}}|F_{n_j}|(2\delta+3\gamma)-|F_n|(2\delta+3\gamma)\\
                 &\ge \sum_{i=1}^k\sum_{F_{n_j}c_j\in {\mathcal{F}}_i}|F_{n_j}|\int\phi \, d\lambda_i-|F_n|(5\delta+7\gamma)\text{ (by \eqref{F-weighted0})}.
 \end{align*}
 For the first term of the righthand side of the above inequation, we have
  \begin{align*}
                 &\quad\sum_{i=1}^k\sum_{F_{n_j}c_j\in {\mathcal{F}}_i}|F_{n_j}|\int\phi \, d\lambda_i\\
                 &\ge \sum_{i=1}^k\bigg(|\cup\mathcal{F}_i|\int\phi \, d\lambda_i-\big(\sum_{F_{n_j}c_j\in {\mathcal{F}}_i}|F_{n_j}|-|\cup\mathcal{F}_i|\big)M\bigg)\\
                 &\ge |F_n|\sum_{i=1}^k\bigg(a_i\int\phi \, d\lambda_i-\frac{3\gamma}{kML|F|}\cdot M \bigg)-\bigg(\sum_{F_{n_j}c_j\in {\mathcal{F}}}|F_{n_j}|-|\cup\mathcal{F}|\bigg)M\text{ (by \eqref{a_i-weighted})}\\
                 &\ge |F_n|\int\phi \, d\lambda-|F_n|5\gamma\text{ (by \eqref{F-weighted00})}.
 \end{align*}
Hence
 \begin{align*}
   S_{F_n}\phi(z)&\ge |F_n|\int\phi \, d\lambda-|F_n|5\gamma-|F_n|(5\delta+7\gamma)\\
                 &> |F_n|(c+\delta-12\gamma)\\
                 &> |F_n|c.
 \end{align*}
 Thus $A_{F_n}\phi(z)> c$, which can deduce that $B_{F_n}(y,\epsilon)\subseteq V_{F_n}$.

\

 Finally,
 \begin{align*}
   \frac{1}{|F_n|}\log \mu(V_{F_n})&\ge \frac{1}{|F_n|}\log \sum_{(x_{D})}\mu(B_{F_n}(y,\epsilon))\\
   &\ge \frac{1}{|F_n|}\log QC\exp(-S_{F_n}\psi(y))\\
   &\ge \frac{1}{|F_n|}\log \exp\bigg(|F_n|\sum_{i=1}^k(a_i-\frac{3\gamma}{kML|F|})(h_{\lambda_i}-4\gamma)\bigg)C\exp(-S_{F_n}\psi(y))\\
   &= \sum_{i=1}^k(a_i-\frac{3\gamma}{kML|F|})(h_{\lambda_i}-4\gamma)-A_{F_n}\psi(y)+\frac{1}{|F_n|}\log C.\\
 \end{align*}

 When $n$ goes to infinity and $\gamma$ goes to $0$, we obtain that
 $$\liminf_{n\rightarrow\infty}\frac{1}{|F_n|}\log \mu(V_{F_n})\ge h_{\nu}(X,G)-\int\psi \, d\nu.$$

\end{proof}

{\bf Acknowledgements}
The authors would like to express their gratitude to the referee for many valuable suggestions and comments. The research was supported by the National Basic Research Program of China (Grant No. 2013CB834100) and the National Natural
Science Foundation of China (Grant No. 11271191, 11431012).

\end{document}